\documentclass[%a4paper, 
11pt,  reqno]{amsart}
\usepackage[margin=1.2in,marginparwidth=1.5cm, marginparsep=0.5cm]{geometry}

\usepackage{amsmath,amssymb,amscd,amsthm,amsxtra, esint}

\usepackage{booktabs} % nice tables
\usepackage{microtype}
\usepackage{amssymb}
\usepackage{mathrsfs}

\usepackage{upgreek} 

\usepackage{color}
\usepackage[implicit=true]{hyperref}

\usepackage{xsavebox}

\usepackage{cases}%%%

\allowdisplaybreaks[2]

\definecolor{gr}{rgb}   {0.,   0.69,   0.23 }
\definecolor{bl}{rgb}   {0.,   0.5,   1. }
\definecolor{mg}{rgb}   {0.85,  0.,    0.85}
\definecolor{yl}{rgb}   {0.8,  0.7,   0.}
\definecolor{or}{rgb}  {0.7,0.2,0.2}

\usepackage{tikz} 
\usepackage{marginnote}
\usepackage{scalerel} % to rescale the paraproduct symbols

\usetikzlibrary{shapes.misc}
\usetikzlibrary{shapes.symbols}
\usetikzlibrary{decorations}
\usetikzlibrary{decorations.markings}

\tikzset{
	dot/.style={circle,fill=black,draw=black,inner sep=0pt,minimum size=0.5mm},
	>=stealth,
	}

\tikzset{
	dot2/.style={circle,fill=black,draw=black,inner sep=0pt,minimum size=0.2mm},
	>=stealth,
	}

\tikzset{
	ddot/.style={circle,fill=black,draw=black,inner sep=0pt,minimum size=0.8mm},
	>=stealth,
	}

\tikzset{decision/.style={ % requires library shapes.geometric
        draw,
        diamond,
        aspect=1.5
    }}

\tikzset{dia2/.style
={diamond,fill=white,draw=black,inner sep=0pt,minimum size=1mm},
	>=stealth,
	}

\tikzset{dia/.style
={star,fill=black,draw=black,inner sep=0pt,minimum size=1mm},
	>=stealth,
	}

\tikzset{dia/.style
={diamond,fill=black,draw=black,inner sep=0pt,minimum size=1.3mm},
	>=stealth,
	}

\makeatletter
\def\DeclareSymbol#1#2#3{\xsavebox{#1}{\tikz[baseline=#2,scale=0.15]{#3}}}
\def\<#1>{\xusebox{#1}}
\makeatother

\newsavebox{\peA}
\newsavebox{\pneA}
\newsavebox{\plA}
\newsavebox{\pgA}
\newsavebox{\pleA}
\newsavebox{\pgeA}
\newsavebox{\pezA}

\savebox{\peA}{\tikz \draw (0,0) node[shape=circle,draw,inner sep=0pt,minimum size=8.5pt] {\scriptsize  $=$};}
\savebox{\pneA}{\tikz \draw (0,0) node[shape=circle,draw,inner sep=0pt,minimum size=8.5pt] {\footnotesize $\neq$};}
\savebox{\plA}{\tikz \draw (0,0) node[shape=circle,draw,inner sep=0pt,minimum size=8.5pt] {\scriptsize $<$};}
\savebox{\pgA}{\tikz \draw (0,0) node[shape=circle,draw,inner sep=0pt,minimum size=8.5pt] {\scriptsize $>$};}
\savebox{\pleA}{\tikz \draw (0,0) node[shape=circle,draw,inner sep=0pt,minimum size=8.5pt] {\scriptsize $\leqslant$};}
\savebox{\pgeA}{\tikz \draw (0,0) node[shape=circle,draw,inner sep=0pt,minimum size=8.5pt] {\scriptsize $\geqslant$};}
\savebox{\pezA}{\tikz \draw (0,0) node[shape=circle,draw,
fill=white, % color = white, 
inner sep=0pt,minimum size=8.5pt]{} ;}

\def \peB{\mathchoice
{\scalebox{.7}{{\usebox{\peA}}}}
{\scalebox{.7}{{\usebox{\peA}}}}
{\scalebox{.7}{{\usebox{\peA}}}}
{}
}

\def \pezB{\mathchoice
{\scalebox{.7}{{\usebox{\pezA}}}}
{\scalebox{.7}{{\usebox{\pezA}}}}
{\scalebox{.7}{{\usebox{\pezA}}}}
{}
}

\newcommand{\pe}{\mathbin{{\peB}}}

\newcommand{\pez}{\mathbin{{\pezB}}}

\usetikzlibrary{decorations.pathmorphing, patterns,shapes}

\DeclareSymbol{dot}{-2.7}{\draw (0,0) node[dot]{};}

\DeclareSymbol{1}{0}{\draw[white] (-.4,0) -- (.4,0); \draw (0,0)  -- (0,1.2) node[dot] {};}
\DeclareSymbol{2}{0}{\draw (-0.5,1.2) node[dot] {} -- (0,0) -- (0.5,1.2) node[dot] {};}
\DeclareSymbol{3}{0}{\draw (0,0) -- (0,1.2) node[dot] {}; \draw (-.7,1) node[dot] {} -- (0,0) -- (.7,1) node[dot] {};}
\DeclareSymbol{30}{-3}{\draw (0,0) -- (0,-1); \draw (0,0) -- (0,1.2) node[dot] {}; \draw (-.7,1) node[dot] {} -- (0,0) -- (.7,1) node[dot] {};}
\DeclareSymbol{31}{-3}{\draw (0,0) -- (0,-1) -- (1,0) node[dot] {}; \draw (0,0) -- (0,1.2) node[dot] {}; \draw (-.7,1) node[dot] {} -- (0,0) -- (.7,1) node[dot] {};}
\DeclareSymbol{32}{-3}{\draw (0,0) -- (0,-1) -- (1,0) node[dot] {}; \draw (0,0) -- (0,-1) -- (-1,0) node[dot] {}; \draw (0,0) -- (0,1.2) node[dot] {}; \draw (-.7,1) node[dot] {} -- (0,0) -- (.7,1) node[dot] {};}
\DeclareSymbol{20}{-3}{\draw (0,0) -- (0,-1);\draw (-.7,1) node[dot] {} -- (0,0) -- (.7,1) node[dot] {};}
\DeclareSymbol{22}{-3}{\draw (0,0.3) -- (0,-1) -- (1,0) node[dot] {}; \draw (0,0.3) -- (0,-1) -- (-1,0) node[dot] {};\draw (-.7,1) node[dot] {} -- (0,0.3) -- (.7,1) node[dot] {};}
\DeclareSymbol{31p}{-3}{\draw (0,0) -- (0,-1) -- (1,0) node[dot] {}; \draw (0,0) -- (0,1.2) node[dot] {}; \draw (-.7,1) node[dot] {} -- (0,0) -- (.7,1) node[dot] {}; 
\draw (0,-1) node{\scaleobj{0.5}{\pez}}; 
\draw (0,-1) node{\scaleobj{0.5}{\pe}}; }
\DeclareSymbol{32p}{-3}{\draw (0,0) -- (0,-1) -- (1,0) node[dot] {}; \draw (0,0) -- (0,-1) -- (-1,0) node[dot] {}; \draw (0,0) -- (0,1.2) node[dot] {}; \draw (-.7,1) node[dot] {} -- (0,0) -- (.7,1) node[dot] {}; 
\draw (0,-1) node{\scaleobj{0.5}{\pez}};
\draw (0,-1) node{\scaleobj{0.5}{\pe}};}
\DeclareSymbol{22p}{-3}{\draw (0,0.3) -- (0,-1) -- (1,0) node[dot] {}; \draw (0,0.3) -- (0,-1) -- (-1,0) node[dot] {};\draw (-.7,1) node[dot] {} -- (0,0.3) -- (.7,1) node[dot] {}; 
\draw (0,-1) node{\scaleobj{0.5}{\pez}};
\draw (0,-1) node{\scaleobj{0.5}{\pe}};}
\DeclareSymbol{21p}{-3}{\draw (0,0.3) -- (0,-1) -- (1,0) node[dot] {}; 
\draw (-.7,1) node[dot] {} -- (0,0.3) -- (.7,1) node[dot] {}; 
\draw (0,-1) node{\scaleobj{0.5}{\pez}};
\draw (0,-1) node{\scaleobj{0.5}{\pe}};}

\DeclareSymbol{21}{-3}{\draw (0,0.3) -- (0,-1) -- (1,0) node[dot] {}; 
\draw (-.7,1) node[dot] {} -- (0,0.3) -- (.7,1) node[dot] {}; 
}

\DeclareSymbol{11p}{-3}{\draw (0,0) node[dot2] {} -- (0,-1) -- (1,0) node[dot] {}; 
\draw (0,0) -- (0,1.2) node[dot] {};  %\draw (-.7,1) node[dot] {} -- (0,0) -- (.7,1) node[dot] {}; 
\draw (0,-1) node{\scaleobj{0.5}{\pez}}; 
\draw (0,-1) node{\scaleobj{0.5}{\pe}}; }
\DeclareSymbol{100}{-3}
{\draw[white] (-.4,0) -- (.4,0); 
\draw (0,0) node[dot2] {} -- (0,-1)  ; 
\draw (0,0) -- (0,1.2) node[dot] {};}  %\draw (-.7,1) node[dot] {} -- (0,0) -- (.7,1) node[dot] {}; 
\usetikzlibrary{mindmap,backgrounds,trees,arrows,decorations.pathmorphing,decorations.pathreplacing,positioning,shapes.geometric,shapes.misc,decorations.markings,decorations.fractals,calc,patterns}

\tikzset{>=stealth',
         cvertex/.style={circle,draw=black,inner sep=1pt,outer sep=3pt},
         vertex/.style={circle,fill=black,inner sep=1pt,outer sep=3pt},
         star/.style={circle,fill=yellow,inner sep=0.75pt,outer sep=0.75pt},
         tvertex/.style={inner sep=1pt,font=\scriptsize},
         gap/.style={inner sep=0.5pt,fill=white}}
\tikzstyle{mybox} = [draw=black, fill=blue!10, very thick,
    rectangle, rounded corners, inner sep=10pt, inner ysep=20pt]
\tikzstyle{boxtitle} =[fill=blue!50, text=white,rectangle,rounded corners]

\tikzstyle{decision} = [diamond, draw, fill=blue!20,
    text width=4.5em, text badly centered, node distance=2.5cm, inner sep=0pt]
\tikzstyle{block} = [rectangle, draw, fill=blue!20,
    text width=2.7cm, text centered, rounded corners, minimum height=3em]

\tikzstyle{line} = [draw, very thick, color=black!50, -latex']

\tikzstyle{cloud} = [draw, ellipse,fill=red!40, 
node distance=2.5cm,
    minimum height=1em]

\tikzstyle{cloud2} = [draw, ellipse,fill=red!30, text=white,text width=10em, node distance=2.5cm, text centered, minimum height=4em]

\tikzstyle{cloud3} = [draw, ellipse, fill=cyan!30, 
node distance=2.5cm,
    minimum height=3em, 
    text width=1.7cm]

\tikzstyle{cloud4} = [draw, ellipse,fill=orange!70, node distance=2.5cm,
   text width=2.1cm,  
           minimum height=3em]

\tikzstyle{cloud5} = [draw, ellipse,fill=red!20, node distance=2.5cm,
    text centered, 
    text width=3.0cm, 
    minimum height=3em]

\tikzstyle{cloud6} = [draw, ellipse,fill=red!20, node distance=2.5cm,
    text width=1.5cm, 
    text centered, 
    minimum height=3em]

\tikzset{
    position/.style args={#1:#2 from #3}{
        at=(#3.#1), anchor=#1+180, shift=(#1:#2)
    }
}
\newtheorem{theorem}{Theorem} [section]

\newtheorem{lemma}[theorem]{Lemma}
\newtheorem{proposition}{Proposition}
\newtheorem{remark}[theorem]{Remark}

\newtheorem{oldtheorem}{Theorem}

\DeclareMathOperator*{\supp}{supp}

\newcommand{\1}{\hspace{0.2mm}\text{I}\hspace{0.2mm}}
\newcommand{\II}{\text{I \hspace{-2.8mm} I} }
\newcommand{\III}{\text{I \hspace{-2.9mm} I \hspace{-2.9mm} I}}
\newcommand{\noi}{\noindent}
\newcommand{\Z}{\mathbb{Z}}
\newcommand{\R}{\mathbb{R}}
\newcommand{\T}{\mathbb{T}}

\let\P= \undefined
\newcommand{\P}{\mathbf{P}}

\newcommand{\K}{\mathcal{K}}

\newcommand{\F}{\mathcal{F}}

\newcommand{\dl}{\delta}

\newcommand{\nb}{\nabla}

\newcommand{\Dl}{\Delta}

\newcommand{\g}{\gamma}

\newcommand{\s}{\sigma}

\newcommand{\ft}{\widehat}

\newcommand{\wt}{\widetilde}

\renewcommand{\l}{\ell}

\newcommand{\les}{\lesssim}
\newcommand{\ges}{\gtrsim}

\newcommand{\jb}[1]
{\langle #1 \rangle}

\newcommand{\ind}{\mathbf 1}

\renewcommand{\S}{\mathcal{S}}

\newcommand{\Sb}{\mathbf{S}}

\newcommand{\M}{\mathcal{M}}

\newcommand{\N}{\mathbb{N}}

\newtheorem*{ackno}{Acknowledgements}

\numberwithin{equation}{section}
\numberwithin{theorem}{section}

\newcommand{\Q}{\mathbf{Q}}

\newcommand{\GG}{\mathcal{G}}

\usepackage{comment}

\newcommand{\Id}{\mathrm{Id}}

\makeatletter
\DeclareRobustCommand\widecheck[1]{{\mathpalette\@widecheck{#1}}}
\def\@widecheck#1#2{%
   \setbox\z@\hbox{\m@th$#1#2$}%
   \setbox\tw@\hbox{\m@th$#1%
      \widehat{%
         \vrule\@width\z@\@height\ht\z@
         \vrule\@height\z@\@width\wd\z@}$}%
   \dp\tw@-\ht\z@
   \@tempdima\ht\z@ \advance\@tempdima2\ht\tw@ \divide\@tempdima\thr@@
   \setbox\tw@\hbox{%
      \raise\@tempdima\hbox{\scalebox{1}[-1]{\lower\@tempdima\box\tw@}}}%
   {\ooalign{\box\tw@ \cr \box\z@}}}
\makeatother

\makeatletter
\@namedef{subjclassname@2020}{%
  \textup{2020} Mathematics Subject Classification}
\makeatother

\begin{document}
\baselineskip = 14pt

\title[Fractional Leibniz rule on the torus]
{Fractional Leibniz rule on the torus}

\author[\'A. B\'enyi,  T. Oh, and T.~Zhao]{\'Arp\'ad B\'enyi, Tadahiro Oh, and Tengfei Zhao}

\address{\'Arp\'ad B\'enyi, Department of Mathematics,
516 High St, Western Washington University, Bellingham, WA 98225,
USA.} \email{benyia@wwu.edu}

\address{
Tadahiro Oh, School of Mathematics\\
The University of Edinburgh\\
and The Maxwell Institute for the Mathematical Sciences\\
James Clerk Maxwell Building\\
The King's Buildings\\
Peter Guthrie Tait Road\\
Edinburgh\\ 
EH9 3FD\\
 United Kingdom\\ and 
School of Mathematics and Statistics, Beijing Institute of Technology, Beijing 100081,
China}

\email{hiro.oh@ed.ac.uk}

\address{
Tengfei Zhao\\
School of Mathematics and Physics\\
University of Science and Technology Beijing\\
 Beijing 100083, China}

\email{zhao$\_$tengfei@ustb.edu.cn}



\begin{abstract}
We discuss the fractional Leibniz rule for periodic functions on the 
 $d$-dimensional torus,
 including the endpoint cases.
 As an application, 
we present a  product estimate, involving distributions
of negative regularities.

\end{abstract}

%\date{\today}
%%

\maketitle

\tableofcontents

\section{Introduction}

In their seminal work \cite{KatoPonce}
 on the well-posedness theory of the Navier-Stokes and Euler equations, Kato and Ponce proved  the following commutator estimate:  %for $s>0$, $1<p<\infty$
%
%In a seminal work , 
%Kato and Ponce proved the following 
% commutator estimate:
\begin{align}
\|J^s(fg)- f J^s g\|_{L^p(\R^d)}
& \les  \|J^s f\|_{L^{p}(\R^d)} 
\|  g\|_{L^{\infty}(\R^d)}
+ 
\| \nb  f\|_{L^{\infty}(\R^d)} \|J^{s-1}g\|_{L^{p}(\R^d)}
\label{KP1} 
\end{align}

\noi
for $s > 0$, $1 < p < \infty$,  and $f, g\in \S(\R^d)$.
%which 
%played an important role in the well-posedness theory of the Navier-Stokes and Euler equations in Sobolev spaces. 
Here, $J^s = J^s_{\R^d} =(1-\Dl)^\frac s2$ denotes 
the Fourier multiplier operator on $\R^d$ 
with multiplier $ (1+ 4\pi^2|\xi|^2)^\frac{s}{2}$.\footnote{In the following, 
when there is no confusion, we simply write $J^s$
to denote the operator $(1-\Dl)^\frac s2$ on either $\R^d$
or $\T^d$, depending on the context.
A similar comment applies to $D^s = (-\Dl)^\frac s2$.}
One variant of the Kato-Ponce inequality~\eqref{KP1} 
takes the form of 
 the 
 fractional Leibniz rule on the Euclidean space $\R^d$, 
 which we now recall.

\begin{oldtheorem}\label{THM:A}
Let $s > 0$,
$1\le p_j,q_j \le \infty$, $j=1,2$, and $\frac 12 \le r \le \infty$ such that $\frac{1}{r}=\frac{1}{p_j}+\frac{1}{q_j}$.
Suppose that one of the following conditions holds\textup{:}

\begin{itemize}
\item[(i)]
\textup{(non-endpoint case).}
Let    $1<p_j,q_j \le \infty$, $j=1,2$, and $\frac 12 < r < \infty$ such that 
%$\frac{1}{r}=\frac{1}{p_j}+\frac{1}{q_j}$
%and  
$s > \frac  dr - d$ or $s \in 2\N$.

\medskip
\item[(ii)]
\textup{($L^\infty$-endpoint case).}
Let $p_j = q_j =  r = \infty$,  $j = 1, 2$.
%\cite{BL}   Also  partial result \cite{GMN}

\medskip
\item[(iii)]
\textup{($L^1$-endpoint case).}
Let  $p_1 = p_2 = 1$, $1 \le q_1 = q_2 \le \infty$, and $\frac 12 \le r \le 1$
such that $\frac 1r = 1 + \frac 1{q_j}$, $j = 1, 2$.
%\cite{OW}

\end{itemize}

\medskip

\noi
Then, the following estimates hold\textup{:}
\begin{align}
\|D^s(fg)\|_{L^r(\R^d)}
& \les\| D^s f\|_{L^{p_1}(\R^d)} \|g\|_{L^{q_1}(\R^d)}+ \|f\|_{L^{p_2}(\R^d)} 
\|  D^s g\|_{L^{q_2}(\R^d)},\label{Leib1} \\
\|J^s(fg)\|_{L^r(\R^d)}
& \les\| J^s f\|_{L^{p_1}(\R^d)} \|g\|_{L^{q_1}(\R^d)}+ \|f\|_{L^{p_2}(\R^d)} 
\|  J^s g\|_{L^{q_2}(\R^d)}
\label{Leib2} 
\end{align}

\noi
for any $f, g \in \S(\R^d)$, 
where  $D^s=
D^s_{\R^d} = 
(-\Dl)^\frac s2$ denotes 
the Fourier multiplier operator on $\R^d$ with multiplier $(2\pi|\xi|)^s$.

\end{oldtheorem}

See 
\cite[Proposition 3.3]{CW}
and \cite[Theorems~A.8 and~A.12]{KPV93}
for the classical non-endpoint results, where 
there are further restrictions such as  $0 < s < 1$
and $1<p_j,q_j, r \lneq \infty$;
see also \cite{GK, Gra0}.
In the non-endpoint case~(i), 
the restriction
$s > \frac  dr - d$ plays no role when $r \ge 1$,
since we assume $s > 0$.
The non-endpoint case~(i) was established
independently by 
Muscalu and Schlag \cite[(2.1)]{MS} (for the one-dimensional case
which can be easily extended to higher dimensions)
and by Grafakos and S.\,Oh \cite{GO}.
We point out that the condition on $s$ in 
the non-endpoint case~(i)
is  sharp
(with an extra assumption
 $p_j,q_j < \infty$, $j=1,2$); see \cite[Theorem~2]{GO}.
 See also \cite[p.\,30]{MS}.
Note that, 
in the non-endpoint case (i), 
the proofs of Theorem~\ref{THM:A} in \cite{MS, GO}
involve (a)~Bony's paraproduct decomposition~\cite{Bony}
and the Coifman-Meyer theorem \cite{CM1, CM2}
or (b)~square function estimates and vector-valued maximal inequalities, 
  and hence they do not extend to the endpoint cases (ii) and~(iii).
In \cite{BL}, 
by introducing a ``low frequency to high frequency switch''
in carrying out summations, Bourgain and Li established
Theorem~\ref{THM:A}
in the $L^\infty$-endpoint case~(ii);
see also a partial result \cite{GMN}.
In the $L^1$-endpoint case~(iii),
when $q_1 = q_2 < \infty$, 
we have $r < 1$, which 
brings  a new difficulty
 since $L^r(\R^d)$
 is only a quasi-Banach space, where Young's inequality 
fails.
By establishing new linear and bilinear  multiplier estimates
 on quasi-Banach spaces, 
S.\,Oh and Wu \cite{OW} overcame this difficulty and
established 
Theorem \ref{THM:A}
in the $L^1$-endpoint case (iii).

The fractional Leibniz rule on $\R^d$ (Theorem \ref{THM:A})
has played a fundamental role in the study 
of nonlinear PDEs on $\R^d$ (or on $\R^d_x \times \R_t$), especially in low regularity setting;
see, for example,~\cite{KPV93}.
Similarly, in studying nonlinear PDEs
in  the periodic setting, i.e.~posed
on the $d$-dimensional torus $\T^d = (\R/\Z)^d$
(or on $\T^d_x \times \R_t$), 
the fractional Leibniz rule on $\T^d$
has played an essential role;
see, for example, 
\cite{IK2, TV, KP, GKO, CLOP}.
Unfortunately, there seems to be no standard 
reference for 
the fractional Leibniz rule on $\T^d$.
In some works, 
special cases  of the fractional Leibniz rule on $\T^d$
were proved in an ad hoc manner
(see, for example,  
\cite[Lemma~9.A.2]{IK2} and \cite[Lemma~3.4]{GKO}), 
whereas some other works
simply invoked 
the fractional Leibniz rule on~$\T^d$
without any proper proof or reference.

Our main goal in this note is 
to provide a simple proof of  the  fractional Leibniz rule on the $d$-dimensional torus $\T^d$, 
corresponding to that on $\R^d$ stated in Theorem~\ref{THM:A}.
%so that it can serve  as a standard point of reference.
In \cite{BO}, 
the first two authors presented an elementary proof of Sobolev's inequality on $\T^d$,
which was ``part of the mathematical analysis folklore'', 
to provide a standard point of reference, especially
to young researchers.
We hope that the following proposition %Proposition \ref{PROP:1}
serves a similar purpose.

\begin{proposition}\label{PROP:1}
Let $s > 0$,
$1\le p_j,q_j \le \infty$, $j=1,2$, and $\frac 12 \le r \le \infty$ such that $\frac{1}{r}=\frac{1}{p_j}+\frac{1}{q_j}$.
Suppose that one of the conditions (i), (ii), and (iii) in 
Theorem \ref{THM:A} holds.
Then, we have 
\begin{align}
\|D^s (fg)\|_{L^r(\T^d)}
& \les\| D^s f\|_{L^{p_1}(\T^d)} \|g\|_{L^{q_1}(\T^d)}+ \|f\|_{L^{p_2}(\T^d)} 
\|  D^s g\|_{L^{q_2}(\T^d)},\label{Leib3} \\
\|J^s(fg)\|_{L^r(\T^d)}
& \les\| J^s f\|_{L^{p_1}(\T^d)} \|g\|_{L^{q_1}(\T^d)}+ \|f\|_{L^{p_2}(\T^d)} 
\|  J^s g\|_{L^{q_2}(\T^d)}
\label{Leib4} 
\end{align}

\noi
for any $f, g \in C^\infty(\T^d)$, 
where  $D^s= D^s_{\T^d}=(-\Dl)^\frac s2$
and  $J^s = J^s_{\T^d}= (1-\Dl)^\frac s2$ denote
the Fourier multiplier operators on $\T^d$ with multipliers $(2\pi |n|)^s$
and $ (1+ 4\pi^2|n|^2)^\frac{s}{2}$, respectively;
namely, 
\begin{align*}
D^s f(x) & = \sum_{n \in \Z^d} (2\pi |n|)^s\ft f(n) e^{2\pi i n \cdot x}, \\
J^s f(x) & =  \sum_{n \in \Z^d} (1+ 4\pi^2|n|^2)^\frac{s}{2}\ft f(n) e^{2\pi i n \cdot x}.
\end{align*}

\end{proposition}

See
\cite[Lemma~9.A.2]{IK2} and \cite[Lemma~3.4]{GKO}, 
where 
special cases of Proposition~\ref{PROP:1} on~$\T^d$
were proven.
In \cite[Lemma~9.A.2]{IK2}, 
Ionescu and Kenig proved  \eqref{Leib4}
with $0 < s < 1$ and $r = p_1 = 2$
(and $q_1 = \infty$)
by transferring the corresponding estimate on $\R^d$
(\cite[Theorems~A.12]{KPV93})
via the Poisson summation formula.
In \cite[Lemma~3.4]{GKO}, 
the bound \eqref{Leib4}
with $0 < s < 1$ 
and 
$1 < p_j,q_j, r <  \infty$
was shown as a corollary to 
\cite[Proposition~3.3]{CW}
and the bilinear transference principle \cite{FS}
(applied to the application of  the Coifman-Meyer theorem \cite{CM1, CM2}).
See also 
\cite[Theorem 3]{Gatto} 
for some non-endpoint fractional Leibniz rule on a space of homogeneous type
(which in particular includes the $d$-dimensional torus).

In \cite{MS}, the fractional Leibniz rule \eqref{Leib1} on $\R^d$
in the non-endpoint case (i) was shown 
via 
 Bony's paraproduct decomposition~\cite{Bony}
and the Coifman-Meyer theorem \cite{CM1, CM2};
see also \cite[Theorem 2.15]{MS}.
As such, 
the corresponding result on $\T^d$ 
in the non-endpoint case follows from 
the  paraproduct decomposition
and the bilinear transference principle~\cite{FS}
applied to each of the paraproducts
(estimated by the Coifman-Meyer theorem).
Note that, instead of invoking the bilinear transference principle, 
it is possible to directly prove
the Coifman-Meyer theorem on $\T^d$;
see
\cite[Problem 3.4]{MS} for such a result in the (more complicated)
bi-parameter setting.\footnote{Interested readers may want to directly 
prove \eqref{Leib3} and \eqref{Leib4} on $\T^d$, at least in the non-endpoint case (i), 
without using a transference result,
by adapting \cite[Chapter 2]{MS} to the periodic setting.}
We point out, however, that 
such an approach via a transference principle 
or the Coifman-Meyer theorem on $\T^d$
does not work in the endpoint cases
(see, for example, 
\cite[Theorem 4.3.7]{Gra1} in the linear case, 
where a restriction $1 < p < \infty$ appears).
In Section \ref{SEC:2}, 
we instead present a proof of Proposition \ref{PROP:1}
as a direct corollary to 
the fractional Leibniz rule~\eqref{Leib1} and~\eqref{Leib2} on $\R^d$, 
using the Poisson summation formula.

\begin{remark}\rm
 A generalization of \eqref{Leib2} in the non-endpoint case to bilinear pseudodifferential operators with symbols in 
 the so-called exotic H\"ormander classes (which encompass the symbols of Coifman-Meyer type) was obtained in \cite[Theorem 2]{BT}; see also \cite{BN,  NT}
  for some weighted versions and applications.  
 Note that  
\cite[Theorem 2]{BT} has a natural counterpart on $\T^d$. 
 Interested readers may consider adapting some of the methods suggested in this note to prove such a result.

\end{remark}

As an application of the fractional Leibniz rule (Theorem \ref{THM:A}
and Proposition \ref{PROP:1}), 
we present the following product estimate, involving distributions
of negative regularities.

\begin{proposition}\label{PROP:2}
Let $\M = \R^d$ or $\T^d$.
Let  $s > 0$ and  $1\le p,q,r\le \infty$ such that 
\begin{align}
 \frac{1}{p}+\frac{1}{q}\le \frac{1}{r}+ \frac sd
\qquad \text{and}\qquad
q, r' \ge p', 
\label{GKO0}
\end{align}

\noi
where $p'$ denotes the H\"older conjugate of $p$, etc., 
and, moreover, when $\M = \R^d$, 
\begin{align}
\frac{1}{p}+\frac{1}{q} \ge \frac 1r .
\label{GKO0a}
\end{align}

\noi
Furthermore,  suppose in addition that one of the following conditions hold\textup{:}

\smallskip
\begin{itemize}
\item[(i)]
 $1 < p \le \infty$ and  $1 < q, r < \infty$,

\smallskip
\item[(ii)]
 $1 < p = r \le \infty$ and  $ q = \infty$,

\smallskip
\item[(iii)]
$1< p \le \infty$, $1 \le q < \infty$,  
and 
 $r = 1$
such that $q = p'$,

\smallskip
\item[(iv)]
$p= r = 1$ and $q = \infty$.

\end{itemize}

\noi
Then, we have
\begin{align}
\|\jb{\nb}^{-s}(fg)\|_{L^r(\M)}
\lesssim\| \jb{\nb}^{-s} f\|_{L^{p}(\M)} \| \jb{\nb}^{s} g\|_{L^{q}(\M)} .
\label{GKO1}
\end{align}

\end{proposition}

In
\cite[Lemma 3.4\,(ii)]{GKO}, 
Proposition \ref{PROP:2} was shown for $0 < s< 1$
and $1<p,q,r<\infty$, 
which we now extend to 
 some endpoint cases and to the range $s\ge 1$.
Proposition \ref{PROP:2} has played a fundamental role
in the recent study of singular stochastic dispersive PDEs; 
see, for example, 
\cite{GKO, OPTz,  GKOT, GKO2, OOT1, OOT2, OTWZ}.
We point out that for such an application, we normally take $p = \infty$
(or $p \gg 1$) and thus the second condition in \eqref{GKO0}
does not (essentially) impose any extra condition.
We present 
a
 proof of Proposition \ref{PROP:2} in Section \ref{SEC:A}.

\section{Preliminary lemmas}

In this section, we state several preliminary lemmas.

We first establish decay estimates for the kernels 
of $D^s = D^s_{\R^d}$ and $J^s = J^s_{\R^d}$ on $\R^d$.
Given  $s >0$, let $\K_s$ and $\GG_s$
be 
the kernels of $D^s$ and $J^s$ on $\R^d$, respectively, given by 
\begin{align}
D^s f = \K_s * f \qquad \text{and}\qquad 
J^s f = \GG_s * f.
\label{decay0}
\end{align}

\noi
The following lemma establishes a (well-known) decay
of these kernels away from the origin.
We present its proof for readers' convenience.

\begin{lemma}\label{LEM:decay}
Let $s > 0$.  Then, given $c_0 > 0$, 
there exists $C_1, C_2 > 0$ such that 
\begin{align}
|\K_s(x)| & \le C_1 |x|^{-d-s}, \label{decay1}\\
 |\GG_s(x)|  & \le C_2 |x|^{-d-s}\label{decay1x}
\end{align}

\noi
for any $x\in \R^d$ with $|x| \ge  c_0$, 
where $C_1$ is independent of $c_0$.

\end{lemma}

\begin{proof}
We first consider the case $0 < s < 2$.
By writing
$D^s = (-\Dl) D^{-2+s}$, 
we have $\K_s = (-\Dl) K_{2-s}$, 
where $K_\s$, $\s > 0$, 
denotes 
 the kernel of the Riesz potential $D^{-\s}$ of order $\s > 0$,  given by 
\begin{align*}
K_\s(x) = C_{d, \s} |x|^{-d + \s}, \quad x \ne 0.
\end{align*}

\noi
By applying $-\Dl$ to $K_{2-s}$, we then have
\begin{align*}
|\K_s(x)| = C_{d, 2-s}\big|(-\Dl) |x|^{-d + 2-s}\big|
\les |x|^{-d-s}, 
\end{align*}

\noi
yielding \eqref{decay1}.

Similarly, by writing $J^s = (1-\Dl)J^{-2+s}$, 
we have 
\begin{align}
\GG_s = (1-\Dl) G_{2-s},
\label{decay1a}
\end{align}

\noi 
where $G_\s$, $\s > 0$, 
denotes the kernel of the Bessel potential $(1-\Dl)^{-\s}$
of order $\s > 0$, given by 
\begin{align}
G_\s(x) = C_{d, \s} \int_0^\infty e^{-t} e^{-\frac{|x|^2}{4t}}t^\frac{\s-d}{2}\frac {dt}{t}.
\label{decay2}
\end{align}

\noi
See \cite[p.\,14]{Gra2}.
Note that $G_\s(x)$ is smooth away from the origin $x = 0$.
For $|x| \ge c_0$, we have
\begin{align}
t + \frac{|x|^2}{4t} \ge t + \frac {c_0^2}{4t}\quad
\text{and}\quad
t + \frac{|x|^2}{4t} \ge |x|, 
\label{decay3}
\end{align}

\noi
where the latter follows from completing a square.
Thus, from \eqref{decay2} and \eqref{decay3}, we have
\begin{align}
|G_{2-s}(x)| \les 
e^{-\frac{|x|}{2}}
 \int_0^\infty e^{-\frac t2} e^{-\frac{c_0^2}{8t}}t^\frac{2-s-d}{2}\frac {dt}{t}
\les e^{-\frac{|x|}{2}}
\label{decay4}
\end{align}

\noi
for $|x| \ge c_0$.
Similarly, by applying $-\Dl$ to $G_{2-s}$ in \eqref{decay2} with \eqref{decay3}, we  have
\begin{align}
|(-\Dl) G_{2-s}(x)|
\les 
e^{-\frac{|x|}{2}}
 \int_0^\infty \Big( \frac 1t + \frac {|x|^2} {t^2}\Big)e^{-\frac t2} e^{-\frac{c_0^2}{8t}}t^\frac{2-s-d}{2}\frac {dt}{t}
\les e^{-\frac{|x|}{3}}
\label{decay5}
\end{align}

\noi
for $|x| \ge c_0$.
Hence, \eqref{decay1x} follows from \eqref{decay1a}, \eqref{decay4}, 
and \eqref{decay5}.

When $s \ge 2$, by writing 
\begin{align*}
D^s = (-\Dl)^{k+1} D^{-2(k+1)+s}
\quad \text{and}\quad  
J^s = (1-\Dl)^{k+1}J^{-2(k+1)+s} 
\end{align*}

\noi
with $k = \big[\frac s2\big]$, where $[\,\cdot \,]$ denotes the integer part, 
we can repeat computations analogous to those presented above
and obtain \eqref{decay1} and \eqref{decay1x}.
We omit the details.  
\end{proof}

%Before proceeding to the proof of Proposition \ref{PROP:1}, 

Next, 
we  recall the Poisson summation formula
(\cite[Theorem~3.2.8]{Gra1}).

\begin{lemma}\label{LEM:Poisson}
Let $f\in C(\R^d)$.
Suppose that 
there exist $C, \dl > 0$ such that 
\[ |f(x)| \leq C(1+|x|)^{-d-\dl}\]

\noi
for any $ x \in \R^d$ 
and 
\[ \sum_{n \in \Z^d} |\ft f(n)| < \infty.\]

\noi
Then, %$f$ and $\ft{f}$ are continuous and, 
for any $x\in \R^d$, we have
\begin{align*}
\sum_{n\in\Z^d} \ft{f}(n)e^{2\pi i n \cdot x} = \sum_{n\in \Z^d} f(x + n).
\end{align*}

\end{lemma}

%Before proceeding to the proof of Proposition \ref{PROP:1}, 
Lastly, let us introduce the Littlewood-Paley projector $\Q_j$ and
a useful commutator estimate (Lemma \ref{LEM:MW}).
Let $\phi:\R \to [0, 1]$ be a smooth  bump function supported on $[-\frac{8}{5}, \frac{8}{5}]$ 
and $\phi\equiv 1$ on $\big[-\frac 54, \frac 54\big]$.
For $\xi \in \R^d$, we set $\Phi_0(\xi) = \phi(|\xi|)$
and 
\begin{align*}
\Phi_{j}(\xi) = \phi\big(\tfrac{|\xi|}{2^j}\big)-\phi\big(\tfrac{|\xi|}{2^{j-1}}\big)
\end{align*}

\noi
for $j \in \N$.
Then, for $j \in \Z_{\geq 0} := \N \cup\{0\}$, 
we define  the Littlewood-Paley projector  $\Q_j$ 
as the Fourier multiplier operator with multiplier $\Phi_j$;
\begin{align}
\F_{\T^d}(\Q_j f)(n) = \Phi_j(n) \F_{\T^d}(f)(n), 
\label{LP0}
\end{align}

\noi
where $\F_{\T^d}$ denotes the Fourier transform on $\T^d$.
Note that we have 
\begin{align*}
\sum_{j = 0}^\infty \Phi_j (\xi) = 1
\end{align*}

\noi
 for each $\xi \in \R^d$. 
Thus, 
we have 
\[ f = \sum_{j = 0}^\infty \Q_j f. \]

\noi
Now, given $s\ge 0$ and $j \in \Z_{\ge 0}$, we define the commutator 
$[2^{-js}D^s_{\T^d}\Q_j, f](g)$ by setting 
\begin{align}
[2^{-js}D^s_{\T^d}\Q_j, f](g) = 2^{-js}D^s_{\T^d}\Q_j(fg) -  f \cdot 2^{-js} D^s_{\T^d}\Q_j g.
\label{MW1}
\end{align}

\noi
Then, we have 
the following commutator estimate.

\begin{lemma}\label{LEM:MW}
Let $s \ge 0$ and $j \in \Z_{\ge 0}$.
Then, given $1 \le p, q, r \le \infty$ with $\frac 1r = \frac 1p + \frac 1q$, 
we have 
\begin{align*}
\big\| [2^{-js}D^s_{\T^d}\Q_j, f](g)\big\|_{L^r(\T^d)}
\les 2^{-j} \| \nb f \|_{L^p(\T^d)}
\|g \|_{L^q(\T^d)}, 
\end{align*}

\noi
uniformly in $j \in \Z_{\ge 0}$.

\end{lemma}

By noting that the multiplier for $2^{-js}D^s_{\T^d} \Q_j$
is given by $\Phi_0(n)$ when $j = 0$ and 
by $\Psi(2^{-j}n)$, 
where $\Psi(n) = |n|^s \Phi_1(2 n)$, 
a straightforward adaptation of the proof of 
\cite[Lemma 2.97]{BCD} to the periodic setting 
yields Lemma \ref{LEM:MW}.
See also \cite[Lemma A.10]{MW}
for an analogous commutator estimate on $\T^d$.

\section{Proof of Proposition \ref{PROP:1}}
\label{SEC:2}

In this section, we present a proof of  Proposition \ref{PROP:1}.

\begin{proof}[Proof of Proposition \ref{PROP:1}]
We only present the proof of \eqref{Leib3} since
\eqref{Leib4} follows in a similar manner, using \eqref{decay1x}
in Lemma \ref{LEM:decay}.
Fix $s > 0$ and $p_j, q_j, r$, $j = 1, 2$  as in the statement
of Theorem \ref{THM:A}.

Let us first introduce some notations.
For functions on $\T^d$, 
let 
$\P_{\ne 0}$ denote the projector to the non-zero frequencies, 
defined by 
$\P_{\ne 0}f = f - \int_{\T^d} f dx$.
Then, set $\P_0 = \Id - \P_{\ne 0}$, 
namely, $\P_0f =  \F_{\T^d}(f) (0)= \int_{\T^d} f dx$, 
where $\F_{\T^d}$ denotes the Fourier transform on $\T^d$.

By noting $D^s_{\T^d} \P_0 f= 0$
and applying H\"older's inequality (with $\text{Vol}(\T^d) = 1$), we have 
\begin{align}
    \begin{split}
&\|D^s_{\T^d}(fg)\|_{L^r(\T^d)}\\  
 & \quad 
 \le \|D^s_{\T^d}( \P_{\ne 0} f \P_{\ne 0} g )\|_{L^r(\T^d)} 
 + \|\P_0 g \cdot D^s_{\T^d} \P_{\ne 0} f  \|_{L^r(\T^d)} \\
& \hphantom{XXX}
 + \|\P_0 f \cdot D^s_{\T^d}  \P_{\ne 0} g \|_{L^r(\T^d)} \\
 & \quad 
 \le
 \|D^s_{\T^d}( \P_{\ne 0} f \P_{\ne 0} g )\|_{L^r(\T^d)} 
+ \|D^s_{\T^d}  f \|_{L^{r}(\T^d)} \|g\|_{L^{1}(\T^d)} \\
& \hphantom{XXX}
+ \|f\|_{L^{1}(\T^d)} \|D^s_{\T^d}   g \|_{L^{r}(\T^d)}\\
 & \quad 
 \le
 \|D^s_{\T^d}( \P_{\ne 0} f \P_{\ne 0} g )\|_{L^r(\T^d)} 
+ \|D^s_{\T^d}  f \|_{L^{p_1}(\T^d)} \|g\|_{L^{q_1}(\T^d)} \\
& \hphantom{XXX}
+ \|f\|_{L^{p_2}(\T^d)} \|D^s_{\T^d}   g \|_{L^{q_2}(\T^d)}.
\end{split}
\label{A1}
\end{align}

\noi
Hence, it suffices to prove \eqref{Leib3}
for mean-zero functions on $\T^d$, 
which will allow us to control the first term on the right-hand side
of \eqref{A1}.

In the following, we fix mean-zero functions 
$f$ and $g$ on $\T^d$ and prove \eqref{Leib3}.
Let 
 $\varphi_1, \varphi_2 \in C^\infty(\R^d; [0,1])$ 
 be periodic, $\varphi_j (x + k) = \varphi_j(x)$, 
 $k \in \Z^d$, $j = 1, 2$
 such that  
\begin{align}
\varphi_1 (x)+\varphi_2(x)
 = 1 
\label{A2}
\end{align}
 for any $x \in \R^d$, 
 $\varphi_1$ is supported on $\bigcup_{m\in \Z^d}\big( \big[ \frac18, \frac{7}{8} \big]^d+m\big)$, 
 and 
$\varphi_2$ is supported on $\bigcup_{m\in \Z^d}\big(  \big[-\frac38,  \frac{3}{8} \big]^d+m\big) $.

Let $F$ be the periodic extension of $f$ on $\R^d$. 
We define functions $G_j$, $j = 1, 2$, 
 on $\R^d$
by setting 
\begin{align}
G_1=  \ind_{[0,1]^d} \cdot \varphi_1 \cdot g\qquad \text{and}
\qquad 
G_2= \ind_{[-\frac12,\frac12]^d}\cdot
 \varphi_2 \cdot g.
\label{A3}
\end{align}

\noi
Then, from \eqref{A2} and \eqref{A3}, we have 
\begin{align}
 \F_{\T^d} (fg)(n)
= 
   \sum_{j=1,2} \F_{\R^d} ( F  G_{j} )(n)
\label{A4}
\end{align}

\noi
for any $n \in \Z^d$, where $\F_{\R^d}$ denotes the Fourier transform on $\R^d$.
Then,  from \eqref{A4} and the Poisson summation formula (Lemma \ref{LEM:Poisson}), 
we have 
\begin{align}
\begin{split}
D^s_{\T^d}(fg)(x)  
&=   \sum_{ n\in \Z^d } (2 \pi |n| )^{s}\F_{\T^d}  (fg)(n)  e^{2\pi i n \cdot x}   \\
&  = \sum_{j=1,2} \sum_{n\in \Z^d} 
(2 \pi |n| )^{s}\F_{\R^d}( F G_j  )(n)    e^{2\pi i n \cdot x}  \\
&   =    \sum_{j=1,2}  \sum_{k\in \Z^d} D^s_{\R^d} (F G_j)(x+k)
\end{split}
\label{A5}
\end{align}

\noi
for any $x\in \T^d$.

\medskip

\noi
$\bullet$ 
{\bf Case 1:}
Let us first consider the case 
$\frac12 \leq r <1$.
Then, from $\l^{r}(\Z^d) \subset \l^1(\Z^d)$
and  a change of variables, we have 
\begin{align}
    \| D^s_{\T^d} (f g)\|_{L^r(\T^d)} \lesssim \sum_{j=1,2} \| D^s_{\R^d} (F G_j)(x)\|_{L^r(\R^d)}.
\label{A6}
\end{align}

Let us introduce two more smooth cutoff functions\footnote{The space $C^\infty(\R^d; B)$
denotes the space of infinitely differentiable functions on $\R^d$ with values in $B$.}
$\psi_1, \psi_2 \in C^\infty(\R^d; [0,1])$ 
such that 
$\psi_{j}$ is supported on  $ \big [\frac 1{10}-\frac{j-1}2,\frac 9{10}-\frac{j-1}2\big]^d$
and 
$\psi_{j}\equiv1$ on  $\big[\frac 19 -\frac{j-1}2,\frac 89-\frac{j-1}2\big]^d$, $j = 1, 2$.
Then, 
by noting $ \psi_j \varphi_j = \varphi_j$, 
it follows
from \eqref{A3} that 
\begin{align}
FG_j= (\psi_j F)G_j .
\label{A7}
\end{align}

\noi
Hence, from \eqref{A6}, \eqref{A7}, and 
the fractional Leibniz rule on $\R^d$ (\eqref{Leib1} in Theorem~\ref{THM:A}), we have
\begin{align}
\begin{split}
\| D^s_{\T^d} (f g)\|_{L^r(\T^d)} 
& \les  \sum_{j=1,2}   \| D^s_{\R^d} (\psi_jF ) \|_{ L^{p_1} (\R^d)} \|G_j\|_{ L^{q_1} (\R^d)}\\
& \quad 
+ \sum_{j=1,2} \|  \psi_jF \|_{ L^{p_2}(\R^d) } \| D^s_{\R^d} G_j\|_{ L^{q_2}(\R^d) }.
\end{split}
\label{A8}
\end{align}

From \eqref{A3}
and the definitions of $\varphi_j$ and $\psi_j$, 
we have 
\begin{align}
 \|G_j\|_{ L^{q_1} (\R^d)}
 \le  \|g\|_{ L^{q_1} (\T^d)}
 \quad \text{and}
 \quad 
\|  \psi_jF  \|_{ L^{p_2}(\R^d) }
 \le
\| f \|_{ L^{p_2}(\T^d) }.
\label{A8a}
\end{align}

\noi
In view of \eqref{A8} and \eqref{A8a}, 
 in order to conclude the proof of \eqref{Leib3}
for mean-zero functions $f$ and $g$ on $\T^d$, 
it suffices to prove 
\begin{align} 
    \begin{split}
\|D^s_{\R^d} ( \psi_jF )\|_{L^p(\R^d)} 
& \les  \|D^s_{\T^d}  f\|_{L^p(\T^d)},\\
 \|D^s_{\R^d}  G_j \|_{L^p(\R^d)} 
 & \les \|D^s_{\T^d}  g\|_{L^p(\T^d)}
\end{split}
\label{A9}
\end{align}

\noi
for any $s > 0$, $1 \leq p \leq \infty  $, and $j=1,2$.

In the following, 
we only 
prove  the first bound in \eqref{A9}, 
since, in view of \eqref{A3},  the second bound follows from a similar computation.
Moreover, we only consider the $j = 1$ case.

Let $\K_s$ be the kernel of 
$D^s_{\R^d}$ as in \eqref{decay0}.
Then, 
with $\supp(\psi_1)\subset  \big [\frac 1{10},\frac 9{10}\big]^d$, 
we have
\begin{align}
\begin{split}
D^s_{\R^d} ( \psi_1F)(x)
& = \int_{[\frac 1{10}, \frac 9{10}]^d} \K_s(x-y) (\psi_1 F)(y) dy\\
& = \ind_{[0, 1)^d}(x) \cdot \int_{[\frac 1{10}, \frac 9{10}]^d} \K_s(x-y) (\psi_1 F)(y) dy\\
& \quad+ \ind_{\R^d \setminus [0, 1)^d}(x) \cdot  \int_{[\frac 1{10}, \frac 9{10}]^d} \K_s(x-y) (\psi_1 F)(y) dy\\
& =: \1(x) + \II(x).
\end{split}
\label{A10}
\end{align}

\noi
As for $\II$, we have $|x - y|\ges 1$
and thus 
by Young's inequality
and Lemma \ref{LEM:decay}
with $s > 0$, 
we have 
\begin{align}
\|\II  \|_{L^{p}(\R^d)} 
& \les
\| \ind_{\{|\,\cdot\,|\ges 1\}}\K_s\|_{L^1(\R^d)}
\|\psi_1 F\|_{L^p(\R^d)}
\le \|f\|_{L^p(\T^d)}.
\label{A11}
\end{align}

\noi
When $1 < p < \infty$, 
since we assume that $f$ is a mean-zero function on $\T^d$, we have
\begin{align}
\|f\|_{L^p(\T^d)}
\le \| J^s_{\T^d}f\|_{L^p(\T^d)}
 \sim 
\| D^s_{\T^d}f\|_{L^p(\T^d)}, 
\label{A12}
\end{align}

\noi
where  the second step follows from the Mihlin multiplier theorem (on $\T^d$)\footnote{Let $1 < p < \infty$.
Then, by
the Mihlin multiplier theorem on $\R^d$ (see Theorem 6.2.7
and (6.2.14) in  \cite{Gra1}), 
we have 
\begin{align}
 \| J^s_{\R^d}h \|_{L^p(\R^d)}
 \sim 
\| D^s_{\R^d}h\|_{L^p(\R^d)}
\label{XX1}
\end{align}

\noi
for any function $h$ on $\R^d$
whose Fourier support is contained in $\{|\xi|> \frac 12\}$.
Now,  let $\eta \in C^\infty(\R^d; [0,1])$ 
such that $\eta(\xi) \equiv 0$ for $|\xi|\le \frac 12$
and $\eta(\xi) \equiv 1$ for $|\xi|\ge 1$, 
and 
 define $b_1$ and $b_2$ by 
\begin{align*}
b_1(\xi) = \eta(\xi) \frac{ (2\pi |\xi|)^s}{(1+ 4\pi^2|\xi|^2)^\frac{s}{2}}
\quad \text{and}
\quad 
b_2(\xi) = \eta(\xi) \frac{(1+ 4\pi^2|\xi|^2)^\frac{s}{2}}{ (2\pi |\xi|)^s}.
\end{align*}

\noi
Then, both $b_1$ and $b_2$ are continuous and thus are regulated everywhere
(in the sense of Definition 4.3.6 in~\cite{Gra1}).
Hence, the second step in
\eqref{A12} for mean-zero functions on $\T^d$
follows from \eqref{XX1} and the transference
(Theorem 4.3.7 in \cite{Gra1}).
}
or the Littlewood-Paley characterization of
the non-homogeneous and homogeneous Sobolev spaces, 
together with the fact that $f$ has mean zero.
When $p = \infty$, 
by choosing $1 < q < \infty$ with $sq > d$, 
it follows from the Sobolev embedding theorem
(see Subsection 2.4 in \cite{BO}), 
the Mihlin multiplier theorem (as in \eqref{A12}),  
and H\"older's inequality (with $\text{Vol}(\T^d) = 1$) that 
\begin{align}
\|f\|_{L^\infty(\T^d)}
\les \| J^s_{\T^d}f\|_{L^q(\T^d)}
 \sim 
\| D^s_{\T^d}f\|_{L^q(\T^d)}
\leq \| D^s_{\T^d}f\|_{L^\infty(\T^d)}.
\label{A13}
\end{align}

\noi
When $p = 1$, by duality, H\"older's inequality, 
and \eqref{A13} (applied to $D^{-s}_{\T^d}g$), we have
\begin{align}
\begin{split}
\|f\|_{L^1(\T^d)}
& = \sup_{\substack{\|g\|_{L^\infty(\T^d)} = 1\\g, \, \text{mean zero}}}\bigg|\int_{\T^d} fg dx\bigg|\\
& = \sup_{\substack{\|g\|_{L^\infty(\T^d)} = 1\\g, \, \text{mean zero}}}\bigg|\int_{\T^d} 
(D^s_{\T^d}f) ( D^{-s}_{\T^d}g ) dx\bigg|\\
& \le
\| D^s_{\T^d}f\|_{L^1(\T^d)}
\cdot \sup_{\substack{\|g\|_{L^\infty(\T^d)} = 1\\g, \, \text{mean zero}}}
\| D^{-s}_{\T^d}g\|_{L^\infty(\T^d)}\\
& \les
\| D^s_{\T^d}f\|_{L^1(\T^d)}
\cdot 
\sup_{\substack{\|g\|_{L^\infty(\T^d)} = 1\\g, \, \text{mean zero}}}
\| g\|_{L^\infty(\T^d)}\\
& \les \| D^s_{\T^d}f\|_{L^1(\T^d)}, 
\end{split}
\label{A14}
\end{align}

\noi
where, in the first step, we used the fact that
\begin{align*}
\int_{\T^d} f g dx = \P_0 g \int_{\T^d} f dx + 
\int_{\T^d} f \P_{\ne 0} g dx 
= \int_{\T^d} f \P_{\ne 0} g dx, 
\end{align*}

\noi
since $f$ has mean zero on $\T^d$.
Putting \eqref{A11}, \eqref{A12}, \eqref{A13}, and \eqref{A14} together, 
we obtain
\begin{align}
\|\II  \|_{L^{p}(\R^d)} \les \| D^s_{\T^d}f\|_{L^p(\T^d)}.
\label{A15}
\end{align}

Next, we treat $\1$ in \eqref{A10}.
By noting 
$\F_{\T^d} (\psi_1 f ) (n) = \F_{\R^d} (\psi_1 F)(n)$
and proceeding as in \eqref{A5}
with the definition of $\1$ in \eqref{A10}, we have
\begin{align}
\begin{split}
& \ind_{[0, 1)^d}(x) \cdot D^s_{\T^d}(\psi_1 f)(x)  
   =  
\1(x) +  \ind_{[0, 1)^d}(x) \sum_{k\in \Z^d\setminus\{0\}}    D^s_{\R^d} (\psi_1 F )(x+k)\\
& \quad = \1(x) + \sum_{k\in \Z^d\setminus\{0\}}
\ind_{[0, 1)^d}(x) \cdot  \int_{[\frac 1{10}, \frac 9{10}]^d} \K_s(x+k-y) (\psi_1 F)(y) dy\\
& \quad =: \1(x) + \III(x)
\end{split}
\label{B1}
\end{align}

\noi
for any $x\in \R^d$, 
where we used the identification $\T^d \cong [0, 1)^d$.

We now estimate the $L^p(\T^d)$-norm of 
$D^{s}_{\T^d}(\psi_1 f)$.
By applying dyadic decompositions, we have 
\begin{align}
\begin{split}
\psi_1 f 
& = \sum_{j =0 }^\infty\Q_j(\psi_1 f) \\
%& = \sum_{j , j_1, j_2 =0 }^\infty\Q_j(\Q_{j_1} \psi_1 \cdot \Q_{j_2} f)\\
& = \sum_{j , j_2 =0 }^\infty\Q_j(\Sb_{j_2} \psi_1 \cdot \Q_{j_2} f)
+ \sum_{j , j_2 =0 }^\infty\Q_j(\Sb_{j_2}^\perp \psi_1 \cdot \Q_{j_2} f)\\
%+   \sum_{\substack{j , j_1, j_2 =0\\j_1 >   j_2 + 2} }^\infty\Q_j(\Q_{j_1} \psi_1 \cdot \Q_{j_2} f)\\
& =: A_1 + A_2, 
\end{split}
\label{LP1}
\end{align}

\noi
where 
\begin{align}
\Sb_{j_2} (\psi_1 ) = \sum_{j_1 = 0}^{j_2- 3} \Q_{j_1} \psi_1
\qquad
\text{and}\qquad 
\Sb_{j_2}^\perp (\psi_1 ) = \sum_{j_1 =j_2- 2}^\infty \Q_{j_1} \psi_1.
\label{LP1a}
\end{align}

\noi
Let us first estimate the first term $A_1$ in \eqref{LP1}.
By noting that we have non-trivial contribution 
for $A_1$ only when $|j - j_2| \le 5$
and that 
$\Q_{j_2}\sum_{j = j_2 - 5}^{j_2+5} \Q_j = \Q_{j_2}$, 
we have 
\begin{align}
\begin{split}
D^s_{\T^d}A_1
& = \sum_{\substack{j, j_2 = 0\\|j - j_2 |\le 5}}^\infty 2^{js} [2^{-js} D^s_{\T^d} \Q_j,
\, \Sb_{j_2} \psi_1]( \Q_{j_2} f)
+ 
\sum_{j_2 = 0}^\infty
 \Sb_{j_2} \psi_1 \cdot D^s_{\T^d} \Q_{j_2} f\\
&  =: A_{11} + A_{12}, 
\end{split}
\label{LP2}
\end{align}

\noi
where 
$[2^{-js}D^s_{\T^d} \Q_j,  f](g)$ is 
as in \eqref{MW1}.
By Minkowski's inequality, 
Lemma \ref{LEM:MW}, and 
Bernstein's inequality (\cite[p.\,333]{TAO}),\footnote{In \cite{TAO}, 
Bernstein's inequalities are  stated on $\R^d$ but the same estimates
hold on $\T^d$.}
we have 
\begin{align}
\begin{split}
\|A_{11}\|_{L^p(\T^d)}
& \les
\sum_{j_2 = 0}^\infty
 2^{j_2s} 
2^{-j_2} \|\nb \Sb_{j_2} \psi_1\|_{L^\infty(\T^d)}\| \Q_{j_2} f\|_{L^p(\T^d)}\\
& \les
 \|\nb  \psi_1\|_{L^\infty(\T^d)}\|D^s_{\T^d} \Q_{j_2} f\|_{L^p(\T^d)}\\
& \le C(\psi_1)
\|D^s_{\T^d}  f\|_{L^p(\T^d)}.
\end{split}
\label{LP3}
\end{align}

\noi
As for the second term $A_{12}$ on the right-hand side of \eqref{LP2}, 
we apply
a ``low frequency to high frequency switch'' as in \cite[(3.1)]{BL}
and write 
\begin{align*}
A_{12}
& = 
 \psi_1  D^s_{\T^d}  f
 + \sum_{j_2 = 0}^\infty
 \Sb_{j_2}^\perp \psi_1 \cdot D^s_{\T^d} \Q_{j_2} f, 
\end{align*}

\noi
where $\Sb_{j_2}^\perp$ is as in \eqref{LP1a}.
Then, 
by H\"older's inequality, Minkowski's inequality,  and 
Bernstein's inequality,
we have 
\begin{align}
\begin{split}
\|A_{12}\|_{L^p(\T^d)}
& \les
 \| \psi_1\|_{L^\infty(\T^d)}\|D^s_{\T^d}  f\|_{L^p(\T^d)}\\
& \quad 
 + \sum_{j_2 = 0}^\infty2^{- j_2}
  \|D  \psi_1\|_{L^\infty(\T^d)}\|D^s_{\T^d}  f\|_{L^p(\T^d)}\\
& \le C(\psi_1)
\|D^s_{\T^d}  f\|_{L^p(\T^d)}.
\end{split}
\label{LP4}
\end{align}

Next, we estimate $A_2$.
From \eqref{LP1} and \eqref{LP1a}, we have 
\begin{align*}
A_2 
& = \sum_{j, j_1 = 0}^\infty 
\sum_{j_2 = 0}^{j_1 + 2}
\Q_j(\Q_{j_1}\psi_1 \cdot \Q_{j_2} f)
= \sum_{j, j_1 = 0}^\infty \Q_j(\Q_{j_1}\psi_1 \cdot \Sb_{j_1+5} f).
\end{align*}

\noi
Note that 
 we have non-trivial contribution 
for $A_2$ only when $j \le j_1 + 10$.
Then, by proceeding as above
and summing over  $j \le j_1 + 10$, 
a crude estimate yields
\begin{align}
\begin{split}
\|D^s_{\T^d}A_{2}\|_{L^p(\T^d)}
& \les
\sum_{\substack{j, j_1 = 0\\j \le j_1 + 10}}^\infty
2^{js} 
 \| \Q_{j_1}\psi_1 \|_{L^\infty(\T^d)}\|  f\|_{L^p(\T^d)}\\
& \les
\sum_{j_1 = 0}^\infty
(j_1+10)
2^{j_1s} 
 \| \Q_{j_1}\psi_1 \|_{L^\infty(\T^d)}\|  f\|_{L^p(\T^d)}\\
& \les
 \| \jb{\nb}^{s+1}\psi_1 \|_{L^\infty(\T^d)}\|  f\|_{L^p(\T^d)}\\
& \le C(\psi_1)
\|D^s_{\T^d}  f\|_{L^p(\T^d)}, 
\end{split}
\label{LP5}
\end{align}

\noi
where, in the last step, we used \eqref{A12}, \eqref{A13}, and \eqref{A14}. 
Hence, putting 
\eqref{LP1}, 
\eqref{LP2}, 
\eqref{LP3}, 
\eqref{LP4}, 
and 
\eqref{LP5} together, we obtain
\begin{align}
\|D^s_{\T^d}(\psi_1 f)\|_{L^p(\T^d)}
\les 
 \|D^s_{\T^d}f\|_{L^p(\T^d)}
\label{B1f}
\end{align}

\noi
for any $s \ge 0$
and $1\le p \le \infty$.

For any $x \in [0, 1)^d$, $y \in [\frac 1{10}, \frac 9{10}]^d$,  and $k\in \Z^d\setminus\{0\}$, 
we have 
\begin{align}
|x+k-y| \ges \max(1, |k|) .
\label{B2}
\end{align}

\noi
Then, 
from \eqref{B1}, \eqref{B1f}, Minkowski's inequality,  and  Lemma \ref{LEM:decay} with \eqref{B2},  we have
\begin{align}
\begin{split}
\|\1  \|_{L^p(\R^d)}
& \les  \|D^s_{\T^d}  f\|_{L^p(\T^d)}
+ \| \III\|_{L^p([0, 1)^d)}\\
& \les  \|D^s_{\T^d}  f\|_{L^p(\T^d)}
+ 
\bigg(\sum_{k\in \Z^d\setminus\{0\}}  |k|^{-d-s}\bigg)
\|f\|_{L^p(\T^d)}\\
& \les  \|D^s_{\T^d}  f\|_{L^p(\T^d)}, 
\end{split}
\label{B3}
\end{align}

\noi
where the last step follows from \eqref{A12}, \eqref{A13}, 
and 
\eqref{A14}.

Hence, from 
\eqref{A10}, 
\eqref{A15}, 
and 
\eqref{B3}, 
we conclude the first bound in \eqref{A9}.
Therefore, we obtain \eqref{Leib3}
for mean-zero functions $f$ and $g$ on $\T^d$ in this case.

\medskip

\noi
$\bullet$ 
{\bf Case 2:} 
Next, we consider the case $1 \le r \le \infty$.
In this case, the bound \eqref{A6} no longer holds.
From \eqref{A5}, we have 
\begin{align}
\begin{split}
\| D^s_{\T^d} (f g)\|_{L^r(\T^d)} 
& \le  \sum_{j=1,2} 
\bigg\|  \sum_{k\in \Z^d} D^s_{\R^d} (F G_j)(x+k)\bigg\|_{L^r([0, 1)^d)}\\
& \le \sum_{j=1,2} 
\|   D^s_{\R^d} (F G_j)(x)\|_{L^r(\R^d)}\\
& \quad + \sum_{j=1,2} 
\bigg\|  \sum_{k\in \Z^d\setminus\{0\}} D^s_{\R^d} (F G_j)(x+k)\bigg\|_{L^r([0, 1)^d)}\\
& = :  A_1 + A_2.
\end{split}
\label{C1}
\end{align}

\noi
As for $A_1$, by proceeding as in Case 1 
with \eqref{A7}, \eqref{A8a}, and \eqref{A9},
we obtain
\begin{align}
A_1 \les 
\| D^s f\|_{L^{p_1}(\T^d)} \|g\|_{L^{q_1}(\T^d)}+ \|f\|_{L^{p_2}(\T^d)} 
\|  D^s g\|_{L^{q_2}(\T^d)}.
\label{C2}
\end{align}

\noi
As for $A_2$, by applying  
Lemma \ref{LEM:decay} with \eqref{B2}, 
and Young's inequality  as in \eqref{B3} (in handling $\III$)
and then applying H\"older's inequality, 
we have
\begin{align}
\begin{split}
A_2 
& \les \| FG_j\|_{L^r([0, 1)^d)}
\le \| fg \|_{L^r(\T^d)}
\le \| f\|_{L^{p_1} (\T^d)} \| g \|_{L^{q_1}(\T^d)}\\
& \les \| D^s_{\T^d} f\|_{L^{p_1} (\T^d)} \| g \|_{L^{q_1}(\T^d)}, 
\end{split}
\label{C3}
\end{align}

\noi
where the last step follows from \eqref{A12}, \eqref{A13}, 
and 
\eqref{A14}.
Hence, from \eqref{C1}, \eqref{C2}, and~\eqref{C3}, we obtain \eqref{Leib3}
for mean-zero functions $f$ and $g$ on $\T^d$ in this case.
This completes the proof
of Proposition~\ref{PROP:1}.
\end{proof}

\section{Proof of Proposition \ref{PROP:2}}
\label{SEC:A}

We conclude this paper by presenting a  proof of  Proposition \ref{PROP:2}.

\begin{proof}[Proof of Proposition \ref{PROP:2}]

The desired bound  \eqref{GKO1}
follows from 
the proof of \cite[Lemma 3.4\,(ii)]{GKO}
together with 
the fractional Leibniz rule (Theorem \ref{THM:A} on $\R^d$
and Proposition \ref{PROP:1} on $\T^d$).
We present details for readers' convenience.

Let $s > 0$.
By duality, 
the fractional Leibniz rule (Theorem \ref{THM:A}
and Proposition \ref{PROP:1}), and Sobolev's inequality
(see also below for endpoint cases),
we have 
\begin{align*}
  \big\| \jb{\nb}^{-s} (fg) \big\|_{L^r} 
  & \le  \sup_{\|\jb{\nb}^s h \|_{L^{r'}} =1}  \bigg|\int fgh \, dx \bigg|\\
 & \le \big\| \jb{\nb}^{-s} f \big\|_{L^p}   \sup_{\| \jb{\nb}^s h \|_{L^{r'}} =1} 
 \big\| \jb{\nb}^{s} (gh) \big\|_{L^{p'}} \\
 &  \les \big\| \jb{\nb}^{-s} f \big\|_{L^p}  
 \sup_{\| \jb{\nb}^s h \|_{L^{r'}} = 1} \Big( \| g \|_{L^{\wt  q}} \big\| \jb{\nb}^s h \big\|_{L^{r'}} 
+ \big\|\jb{\nb}^s g \big\|_{L^q} \| h \|_{L^{\wt r'} } \Big) \\
 &  \lesssim   \big\| \jb{\nb}^{-s} f \big\|_{L^p} \big\| \jb{\nb}^s g \big\|_{L^q}   , 
\end{align*}

\noi
where the exponents satisfy the H\"older relations: 
\begin{align}
\frac{1}{p'}  =  \frac1q + \frac1{\wt r'} = \frac1{\wt q} + \frac1{r'}
\label{bilinear3}
\end{align}

\noi
and 
the exponents satisfy the Sobolev relations: 
\begin{align}
\frac{s}{d}   \ge   \frac1q  - \frac1{\wt q}
  \qquad \text{and}\qquad
\frac{s}d \ge  \frac1{r'} -   \frac1{\wt r'}.
\label{bilinear4}
\end{align}

\noi
When $\M = \R^d$, we need $q \le \wt q$ and $r' \le \wt r'$, 
which is guaranteed by \eqref{GKO0a} and \eqref{bilinear3}.
On $\T^d$, we do not need such conditions thanks to the nestedness
of the Lebesgue spaces.

Let us now examine possible ranges of $p$, $q$, and $r$ more closely.
\begin{itemize}
\item[(i)]
When $1 < p \le \infty$, namely $1\le p' < \infty$, 
there is no issue for $1 < q, r < \infty$, 
provided that $q, r' \ge p'$ coming from \eqref{bilinear3}.
Furthermore,

\begin{itemize}
\smallskip
\item[(i.a)] 
When $q = \infty$, 
the application of Sobolev's inequality 
\begin{align}
\| g \|_{L^{\wt  q}}
\les \big\|\jb{\nb}^s g \big\|_{L^q}
\label{Sob1}
\end{align}

\noi
is prohibited.  However, from the nestedness
of Sobolev spaces (recall $s > 0$),\footnote{In order to show the embedding 
$W^{s_1, \infty} \subset W^{s_2, \infty}$
for $s_1 > s_2$, we can not use the Mihlin multiplier theorem.
Instead,  this embedding 
follows from 
\[ W^{s_1, \infty} \subset B^{s_1}_{\infty, \infty} \subset B^{s_2}_{\infty, 1} \subset W^{s_2, \infty}, \]

\noi
where $B^s_{p, q}$ denotes the Besov spaces.} \eqref{Sob1} holds for $\wt q =  q = \infty$,
 which yields $p' = \wt r' = r'$ in view of \eqref{bilinear3}, 
 allowing the endpoint case $r = \infty$.

\smallskip
\item[(i.b)] 
When $q = 1$, we must have $p' = 1$, namely $p = \infty$.
Due to the failure of Sobolev's inequality at the endpoint, 
we must have $\wt q = q = 1$ such that \eqref{Sob1} holds by the nestedness of Sobolev spaces
(and the endpoint fractional Leibniz rule is applicable).
This forces $r' = \wt r' = \infty$, namely $r = 1$.

\smallskip
\item[(i.c)] 
When $r' = \infty$ (namely $r = 1$), 
we must have $\wt r' = r' = \infty$ (for the reason explained in (i.a) when $q= \infty$).
In this case, in view of \eqref{bilinear3}, we must have $p' = q = \wt q$
 allowing the endpoint case $q = 1$.
Note that this case contains Case (i.b) as a subcase.

\smallskip
\item[(i.d)] 
When $r' = 1$ (namely $r = \infty$), 
we must have $\wt r' = r' = 1$, which implies $p ' = 1$
and $q = \wt q = \infty$.
Note that this case is contained in Case (i.a) as a subcase.

\end{itemize}

\smallskip
\item[(ii)] 
When $p = 1$, namely $p' = \infty$, 
we must have $q = \wt q = r' = \wt r' = \infty$, in particular $r = 1$.

\end{itemize}

\noi
Once $p, q, r$ satisfy one of the conditions above, in view of \eqref{GKO0}, we see that there exist 
$\wt q$ and $\wt r'$ satisfying \eqref{bilinear3} and \eqref{bilinear4}.
\end{proof}

\begin{ackno}\rm
 The authors would like to thank an anonymous referee for the helpful comments which have improved the presentation of the paper.
\'A.B.  acknowledges the support from
 an  AMS-Simons Research Enhancement Grant for PUI Faculty.
T.O.~was supported by the European Research Council (grant no.~864138 ``SingStochDispDyn"). 
T.Z.~was supported by the National Natural Science Foundation of China
 (Grant No. 12101040, 12271051, and 12371239) and by a grant from the
China Scholarship Council (CSC). T.Z.~would like to thank the University
of Edinburgh for its hospitality where this manuscript was prepared.

\end{ackno}

\end{document}